\documentclass{amsart}
\usepackage{amssymb}
\usepackage{amsmath}
\usepackage{amsfonts}

\newtheorem{theorem}{Theorem}
\theoremstyle{plain}

\newtheorem{case}{Case}

\newtheorem{corollary}{Corollary}

\newtheorem{lemma}{Lemma}

\newtheorem{proposition}{Proposition}

\numberwithin{equation}{section}
\input{tcilatex}

\begin{document}
\title{Biharmonic Pseudo-Riemannian submersions from $3$-manifolds}
\author{\.{I}rem K\"{u}peli Erken}
\address{Faculty of Natural Sciences, Architecture and Engineering,
Department of Mathematics, Bursa Technical University, Bursa, TURKEY}
\email{irem.erken@btu.edu.tr}
\author{Cengizhan Murathan}
\address{Art and Science Faculty,Department of Mathematics, Uludag
University, 16059 Bursa, TURKEY}
\email{cengiz@uludag.edu.tr}
\date{23.01.2017}

\begin{abstract}
We classify the pseudo-Riemannian biharmonic\ submersion from a\ $3$%
-dimensional space form into a surface.
\end{abstract}

\maketitle

\section{INTRODUCTION}

\bigskip The theory of Riemannian submersions was initiated by O'Neill \cite%
{oneill} and Gray \cite{gray}. One of the well known example of a Riemannian
submersion is the projection of a Riemannian product manifold on one of its
factors. Presently, there is an extensive literature on the Riemannian
submersions with different conditions imposed on the total space and on the
fibres. A systematic exposition could be found in A. Besse's book \cite{BES}%
. Pseudo-Riemannian submersions were introduced by O'Neill \cite{ONEILL2}.
Magid classified pseudo-Riemannian submersions with totally geodesic fibres
from an anti-de Sitter space onto a Riemannian manifold \cite{magid}. Then B%
\u{a}di\c{t}ou gave the classification of the pseudo-Riemannian submersions
with (para) complex connected totally geodesic fibres from a (para) complex
pseudo-hyperbolic space onto a pseudo Riemannian manifold \cite{Ba1, Ba2}.

A map between Riemannian manifolds is harmonic if the divergence of its
differential vanishes. The first major study of harmonic maps has been begun
by J. Eells and J. H. Sampson \cite{ES}. In \cite{ES}, Eells and Sampson
defined biharmonic maps between Riemannian manifolds as an extension of
harmonic maps and Jiang obtained their first and second variational formulas 
\cite{JI}.

During the last decade important progress has been made in the study of both
the geometry and the analytic properties of biharmonic maps. A fundamental
problem in the study of biharmonic maps is to classify all proper biharmonic
maps between certain model spaces. An example of this was proved
independently by Chen-Ishikawa \cite{CHENNN} and Jiang \cite{JI} that every
biharmonic surface in a Euclidean $3$-space $E^{3}$ is a minimal surface.
Later, Caddeo et al. showed that the theorem remains true if the target
Euclidean space is replaced by\ $3$-dimensional hyperbolic space form \cite%
{CAD}. Chen and Ishikawa also proved that biharmonic Riemannian surface in $%
E_{1}^{3}$ is a harmonic surface \cite{CHEN}.\ For Riemannian submersions,
Wang and Ou stated that Riemannian \ submersion from a\ $3$-dimensional
space form into a surface is biharmonic if and only if it is harmonic \cite%
{ze-ye}.

The above results give us the motivation for preparing this study. In this
paper, we study the biharmonic pseudo-Riemannian submersions from $3$%
-manifolds.

The main purpose of section $\S 2$ is to give a brief information about
pseudo-Riemannian submersions, biharmonic maps and space forms. In this
section, we also give some properties of fundamental tensors and fundamental
equations which we will use them to obtain our results. In section $\S 3$,
we investigate the biharmonicity of a pseudo-Riemannian submersion from a $3$%
-manifold by using \ the integrability data of a special orthonormal frame
adapted to a \ pseudo-Riemannian submersion. Finally, we give a complete
classification of biharmonic pseudo-Riemannian submersions from a $3$%
-dimensional pseudo-Riemannian space form.

\section{PRELIMINARIES}

\subsection{Pseudo-Riemannian submersions with totally geodesic fibre}

In this subsection we recall several notions and results which will be
needed throughout the paper.

Let $(M,g)$ be an $m$-dimensional connected pseudo-Riemannian manifold of
index $s$ $(0\leq s\leq m)$, let $(B,g^{\prime })$ be an $n$-dimensional
connected pseudo-Riemannian manifold of index $r\leq s,(0\leq r\leq n)$. In
case of Riemannian submersion, the fibers are always Riemannian manifolds.

A pseudo-Riemannian submersion is a smooth map $\pi :M\rightarrow B$ which
is onto and satisfies the following three axioms:

$S1$. $\pi _{\ast }\mid _{p}$ is onto for all $p\in M$,

$S2$. the restriction of the metric to the fibres $\pi ^{-1}(b)$, $b\in B$
are non degenerate ,

$S3$. $\pi _{\ast }$ preserves scalar products of vectors normal to fibres.

We shall always assume that the dimension of the fibres dim$M$ - dim$B$ is
positive and the fibres are connected. By S2, one can observe fibres as
spacelike and timelike cases.

The vectors tangent to fibres are called vertical and those normal to fibres
are called horizontal. We denote by $V$ the vertical distribution and by $H$
the horizontal distribution. The fundamental tensors of a submersion were
defined by O'Neill (\cite{oneill}, \cite{ONEILL2}). They are $(1,2)$-tensors
on $M$, given by the formulas:%
\begin{eqnarray}
T(E,F) &=&T_{E}F=h\nabla _{\nu E}\nu F+\nu \nabla _{\nu E}hF,  \label{AT} \\
A(E,F) &=&A_{E}F=\nu \nabla _{hE}hF+h\nabla _{hE}\upsilon F,  \notag
\end{eqnarray}%
for any $E,$ $F\in X(M).$ Here $\nabla $ denotes the Levi-Civita connection
of $(M,g).$ These tensors are called integrability tensors for the
pseudo-Riemannian submersions. We use the $h$ and $\nu $ letters to denote
the orthogonal projections on the vertical and horizontal distributions
respectively. A vector field $X$ on $M$ is said to be basic if \ it is the
unique horizontal lift of a vector field $X_{\ast }$ on $B$, so that $\pi
_{\ast }(X)=X_{\ast }$ $\ $is horizontal and $\pi $-related to a vector
field $X_{\ast }$ on $B$. It is easy to see that every vector field $X_{\ast
}$ on $B$ has a unique horizontal lift $X$ to $M$ and $X$ is basic. The
following\ lemmas are well known (see \cite{oneill}, \cite{ONEILL2}).

\begin{lemma}
Let $\pi :(M,g)\rightarrow (B,g^{\prime })$ be a pseudo-Riemannian
submersion. If $\ X,$ $Y$ are basic vector fields on $M$, then
\end{lemma}

$i)$ $g(X,Y)=g^{\prime }(X_{\ast },Y_{\ast })\circ \pi ,$

$ii)$ $h[X,Y]$ is basic and $\pi $-related to $[X_{\ast },Y_{\ast }]$,

$iii)$ $h(\nabla _{X}Y)$ is a basic vector field corresponding to $\nabla
_{X_{\ast }}^{^{B}}Y_{\ast }$ where $\nabla ^{B}$ is the connection on $B.$

$iv)$ for any vertical vector field $V$, $[X,V]$ is vertical.

\begin{lemma}
For any $U,W$ vertical and $X,Y$ horizontal vector fields, the tensor fields 
$T$ and $A$ satisfy
\end{lemma}

$i)T_{U}W=T_{W}U$,

$ii)A_{X}Y=-A_{Y}X=\frac{1}{2}\nu \left[ X,Y\right] .$

Moreover, if $X$ is basic and $U$ is vertical then $h(\nabla _{U}X)=h(\nabla
_{X}U)=A_{X}U.$ Notice that $T$ \ acts on the fibres as the second
fundamental form of the submersion and restricted to vertical vector fields
and it can be easily seen that $T=0$ is equivalent to the condition that the
fibres are totally geodesic.

We define the curvature tensor $R$ of $M$ by $R(E,F)=\nabla _{E}\nabla
_{F}-\nabla _{F}\nabla _{E}-\nabla _{\lbrack E,F]}$ for any vector fields $E$%
, $F$ on $M$. The pseudo-Riemannian curvature $(0,4)$-tensor is defined by%
\begin{equation*}
R(E,F,G,H)=g(R(E,F)G,H).
\end{equation*}

Let us recall the sectional curvature of pseudo-Riemannian manifolds for
nondegenerate planes. Let $M$ be a pseudo-Riemannian manifold and $P$ be a
non-degenerate tangent plane to $M$ at $p$. The number

\begin{equation*}
K_{X\wedge Y}=\frac{g(R(X,Y)Y,X)}{g(X,X)g(Y,Y)-g(X,Y)^{2}}
\end{equation*}%
is independent on the choice of basis $X,Y$ for $P$ and is called the
sectional curvature. We use notation $R_{ijkl}$ $%
=g(R(e_{i},e_{j})e_{k},e_{l}).$ Next, we can give the following lemma:

\begin{lemma}[\protect\cite{ONEILL2}]
Let $\pi :(M,g)\rightarrow (B,g^{\prime })$ be a pseudo-Riemannian
submersion. $K$ and $K^{B}$ denote the sectional curvatures\ in $M$ and $B$,
respectively. If $\ X,$ $Y$ are basic vector fields on $M,$ then%
\begin{equation}
K_{X_{\ast }\wedge Y_{\ast }}^{B}=K_{X\wedge Y}+\frac{3g(A_{X}Y,A_{X}Y)}{%
g(X,X)g(Y,Y)-g(X,Y)^{2}}.  \label{K...}
\end{equation}
\end{lemma}

In \cite{Es}, Escobales gave a classification of Riemannian submersions with
connected totally geodesic fibres from a sphere to a Riemannian manifold and
then Ranjan \cite{Ra} dropped Escobales's classification into three
categories: $(a)$ $S^{2n+1}\rightarrow CP^{n},n\geq 1,$ with the fibres $%
S^{1};$ $(b)$ $S^{4n+3}\rightarrow HP^{n}$,$n\geq 1,$ with the fibres $%
S^{3}; $ $(c)$ $S^{8n+7}\rightarrow CaP^{n}$, $n=1,2$ with the fibres $S^{7}$%
, where $CP^{n}$, $HP^{n}$ and $CaP^{n}$ are complex projective,
quaternionic projective and Cayley projective space, respectively.

In the Lorentz space case,\ Magid \cite{magid} proved that if $\pi
:H_{1}^{2n+1}(c)\rightarrow B^{2n}$ be a pseudo-Riemannian submersion with
totally geodesic fibres onto\ a Riemannian manifold then, $B^{2n}$ is a
Kaehler manifold holomorphically isometric to complex hyperpolic space $%
CH^{n}(4c).$

In \cite{Ba} Baditou and Ianu\c{s} generalized Magid's result and classified
the pseudo-Riemannian submersions with connected complex totally geodesic
fibres from a complex pseudo hyperbolic space onto a Riemannian manifold.
These pseudo-Riemannian submersions are observed as mainly three categories
: $(1)$ $H_{1}^{2m+1}\rightarrow 
\mathbb{C}
H^{m},$ $(2)$ $H_{3}^{4m+3}\rightarrow H(H^{m})$ or $(3)$ $%
H_{7}^{15}\rightarrow H^{8}(-4)$, where $%
\mathbb{C}
H^{m}$and $H(H^{m})$ are complex hyperbolic space and \ quaternionic
hyperbolic space, respectively. Then Baditoiu \cite{Ba1} improved these
results under the assumption that the dimension of the fibres is less than
or equal to three.

Recently, Baditoiu \cite{Ba2} generalized previous results without any
assumption for dimension of the fibres and proved that any pseudo-Riemannian
submersions with connected, totally geodesic fibres from a real pseudo
hyperbolic space onto a pseudo-Riemannian manifold is equivalent to one of
the (para) Hopf pseudo-Riemannian submersions: $\ (i)$ $H_{2t+1}^{2m+1}%
\rightarrow 
\mathbb{C}
H_{t}^{m},0\leq t\leq m,$ $(ii)$ $H_{m}^{2m+1}\rightarrow AP^{m},$ $(iii)$ $%
H_{4t+3}^{4m+3}\rightarrow H(H_{t}^{m}),0\leq t\leq m,$ $(iv)$ $%
H_{2m+1}^{4m+3}\rightarrow BP^{m},$ $(v)$ $H_{15}^{15}\rightarrow
H_{8}^{8}(-4),$ $(vi)$ $H_{7}^{15}\rightarrow H_{4}^{8}(-4)$ or $(vii)$ $%
H_{7}^{15}\rightarrow H_{4}^{8}(-4),$ where $%
\mathbb{C}
H_{t}^{m}$ and $H(H_{t}^{m})$ are the indefinite complex and quaternionic
pseudo-hyperbolic spaces of holomorphic, respectively, quaternionic
curvature $-4$; $AP^{m}$ is the para-complex projective space of real
dimension $2m$, signature $(m,m)$ and para-holomorphic curvature $-4$; $%
BP^{m}$ is the para-quaternionic projective space of real dimension $4m$,
signature $(2m,2m)$ and para-quaternionic curvature $-4$.

In summary, for three dimensional, these (para) pseudo-Riemannian
submersions with connected, totally geodesic fibres\ fall into one of the
following cases:

$(a_{1})$ $\pi :$ $S^{3}(1)\rightarrow CP^{1}$, $(a_{2})$ $\pi :$ $%
H_{1}^{3}(-1)\rightarrow H^{2}(-4)=CH^{1},$ $(a_{3})$ $\pi
:H_{1}^{3}(-1)\rightarrow H_{1}^{2}(-4)=AH^{1}$, $(a_{4})$ $\pi
:H_{3}^{3}(-1)\rightarrow H_{2}^{2}(-4)=CH_{1}^{1}$

We will finish this subsection by the following Theorem of \ Uniqueness:

\begin{theorem}[\protect\cite{Ba2}]
\bigskip Let $\pi _{1},\pi _{2}:H_{l}^{a}\rightarrow B$ be two
pseudo-Riemannian submersions with connected, totally geodesic fibres from a
pseudo-hyperbolic space onto a pseudo-Riemannian manifold. Then there exists
an isometry $f:H_{l}^{a}\rightarrow H_{l}^{a}$ such that $\pi _{2}\circ
f=\pi _{1}$. In particular, $\pi _{1}$ and $\pi _{2}$ are equivalent.
\end{theorem}

\subsection{Biharmonic maps}

Let $M^{m}$ and $B^{n}$ be pseudo-Riemannian manifolds of dimensions $m$ and 
$n$, respectively, and $\varphi :$ $M^{m}\rightarrow B^{n}$ a smooth map. We
denote by $\nabla ^{M}$ and $\nabla ^{B}$ the Levi-Civita connections on $%
M^{m}$ and $B^{n}$, respectively. Then the tension field $\tau (\varphi )$
is a section of the vector bundle $\varphi ^{\ast }TB^{n}$ defined by%
\begin{equation*}
\tau (\varphi )=\text{trace}(\nabla ^{\varphi }d\varphi
)=\dsum\limits_{i=1}^{m}g(e_{i},e_{i})(\nabla _{e_{i}}^{\varphi }d\varphi
(e_{i})-d\varphi (\nabla _{e_{i}}e_{i})).
\end{equation*}%
Here $\nabla ^{\varphi }$ and $\left\{ e_{i}\right\} $ denote the induced
connection by $\varphi $ on the bundle $\varphi ^{\ast }TB^{n}$, which is
the pull-back of $\nabla ^{B}$, and a local orthonormal frame field of $%
M^{m} $, respectively. A smooth map $\varphi $ is called a harmonic map if
its tension field vanishes. A map $\varphi $ is called biharmonic if it is a
critical point of the energy 
\begin{equation*}
E_{2}(\varphi )=\frac{1}{2}\int_{\Omega }g(\tau (\varphi ),\tau (\varphi
)dv_{g}
\end{equation*}%
for every compact domains $\Omega $ of $M^{m}$, where $dv_{g}$ is the volume
form of $M^{m}.$ Using same argument in Riemannian case, the bitension field
can be defined by

\begin{equation}
\tau _{2}(\varphi )=\dsum\limits_{i=1}^{m}g(e_{i},e_{i})((\nabla
_{e_{i}}^{\varphi }\nabla _{e_{i}}^{\varphi }-\nabla _{\nabla
_{e_{i}}e_{i}}^{\varphi })\tau (\varphi )-R^{B}(d\varphi (e_{i}),\tau
(\varphi ))d\varphi (e_{i})),  \label{1}
\end{equation}%
where $R^{B}$ is the curvature tensor of $B^{n}$ (see \cite{DONGOU}, \cite%
{JI}, \cite{sasahara}). A smooth map $\varphi $ is a biharmonic map (or $2$%
-harmonic map) if its bitension field vanishes (see \cite{JI}, \cite%
{sasahara}). By definition, a harmonic map is clearly biharmonic map. Non
harmonic maps are called proper biharmonic maps.

\section{THE THEOREMS AND PROOFS}

In this section, we will prove our classification Theorem and corollaries.
Firstly, we will recall well known theorems:

\begin{theorem}[\protect\cite{falcitelli}]
A pseudo-Riemannian submersion $\pi :(M,g)\rightarrow (B,g^{^{\prime }})$ is
a harmonic map if and only if each fibre is a minimal submanifold.
\end{theorem}

\begin{theorem}[\protect\cite{Ba1},\protect\cite{magid},\protect\cite{Ra},%
\protect\cite{Es}]
Let $\pi :(M_{r}^{3}(c),g)\rightarrow (B_{s}^{2},g^{^{\prime }})$ be a
(para) pseudo-Riemannian submersion with connected totally geodesic fibres,
where $0\leq r\leq 3,$ $0\leq s\leq 2$ and $c\neq 0.$In summary, for three
dimensional, these (para) pseudo-Riemannian submersions with connected,
totally geodesic fibres. Then $\pi $ is one of the following types:

\begin{tabular}{|l|l|}
\hline
\ \ \ \ \ \ \ \ \ \ Timelike Fiber & \ \ \ \ \ \ \ \ \ \ \ \ \ \ \ \ \ \ \ \
\ Spacelike Fiber \\ \hline
$H_{3}^{3}(-1)\overset{\pi }{\rightarrow }H_{2}^{2}(-4)=CH_{1}^{1};$\cite%
{Ba1} & $H_{1}^{3}(-1)\overset{\pi }{\rightarrow }H_{1}^{2}(-4)=AH^{1};$\cite%
{Ba1} \\ \hline
$H_{1}^{3}(-1)\overset{\pi }{\rightarrow }H^{2}(-4)=CH^{1};$\cite{magid} & $%
S^{3}(1)\overset{\pi }{\rightarrow }S^{2}\left( \frac{1}{2}\right) =CP^{1};$%
\cite{Ra},\cite{Es}. \\ \hline
\end{tabular}
\end{theorem}

We will report following theorems which give us the motivation to study on
this paper.

\begin{theorem}[\protect\cite{CHEN}]
Let $\ x:M\rightarrow E_{s}^{3}$ $(s=0,1)$ be a biharmonic isometric
immersion of a Riemannian surface $M$ into $E_{s}^{3}$ $.$Then $x$ is
harmonic.

\begin{theorem}[\protect\cite{ZHANG}]
If $M$ is a complete biharmonic space-like surface in $S_{1}^{3}$ or $%
R_{1}^{3},$ then it must be totally geodesic, i.e. $S^{2}$ or $R^{2}.$
\end{theorem}
\end{theorem}

\begin{theorem}[\protect\cite{ze-ye}]
Let $\pi :(M^{3}(c),g)\rightarrow (B^{2},g^{^{\prime }})$ be a Riemannian
submersion from a space form of constant sectional curvature $c$. Then, $\pi 
$ is biharmonic if and only if it is harmonic, and this holds if and only if
it is a harmonic morphism.
\end{theorem}

Let $\pi :(M_{r}^{3},g)\rightarrow (B_{s}^{2},g^{^{\prime }})$ be a
pseudo-Riemannian submersion where $0\leq r\leq 3,$ $0\leq s\leq 2$. \ Let
us consider a local pseudo orthonormal frame $\{e_{1},e_{2},e_{3}\}$ such
that $e_{1},e_{2}$ are basic and $e_{3}$ is vertical . Then, it is well
known (see \cite{oneill}) that $\left[ e_{1},e_{3}\right] $ and $\left[
e_{2},e_{3}\right] $ are vertical and $\left[ e_{1},e_{2}\right] $ is $\pi $%
-related to $\left[ \varepsilon _{1},\varepsilon _{2}\right] $, where $%
\left\{ \varepsilon _{1},\varepsilon _{2}\right\} $ is a pseudo orthonormal
frame in the base manifold.

Let $\{e_{1},e_{2},e_{3}\}$ be an orthonormal frame adapted to with $e_{3}$
being vertical where $g(e_{i},e_{i})=\delta _{i}=\mp 1.$ If we assume that%
\begin{equation}
\left[ \varepsilon _{1},\varepsilon _{2}\right] =L_{1}\varepsilon
_{1}+L_{2}\varepsilon _{2},  \label{2}
\end{equation}%
for $L_{1},$ $L_{2}\in C^{\infty }(B)$ and use the notations $%
l_{i}=L_{i}\circ \pi ,$ $i=1,2.$ Then, we have%
\begin{eqnarray}
\left[ e_{1},e_{3}\right] &=&\lambda e_{3,}  \notag \\
\left[ e_{2},e_{3}\right] &=&\mu e_{3,}  \label{3} \\
\left[ e_{1},e_{2}\right] &=&l_{1}e_{1}+l_{2}e_{2}-2\sigma e_{3.}  \notag
\end{eqnarray}%
where $\lambda ,$ $\mu $ and $\sigma $ $\in C^{\infty }(M).$ Here $l_{1}$, $%
l_{2}$, $\lambda $, $\mu $ and $\sigma $ are the integrability functions of
the adapted frame of the pseudo-Riemannian submersion $\pi .$

\begin{proposition}
Let $\pi :(M_{r}^{3},g)\rightarrow (B_{s}^{2},g^{^{\prime }})$ be a
pseudo-Riemannian submersion with the adapted frame $\left\{
e_{1},e_{2},e_{3}\right\} $ and the integrability functions $l_{1}$, $l_{2}$%
, $\lambda $, $\mu $ and \ $\sigma .$ Then, the pseudo-Riemannian submersion 
$\pi $ is biharmonic if and only if%
\begin{eqnarray}
\Delta ^{M}\lambda &=&-\delta _{2}l_{1}e_{1}(\mu )-\delta _{2}e_{1}(\mu
l_{1})-\delta _{2}l_{2}e_{2}(\mu )-\delta _{2}e_{2}(\mu l_{2})  \notag \\
&&+\delta _{2}\lambda \mu l_{1}+\delta _{2}\mu ^{2}l_{2}+\lambda \left\{
\delta _{2}l_{1}^{2}+\delta _{1}l_{2}^{2}-\delta _{1}\delta
_{2}K^{B}\right\} ,  \label{4} \\
\Delta ^{M}\mu &=&\delta _{1}l_{1}e_{1}(\lambda )+\delta _{1}e_{1}(\lambda
l_{1})+\delta _{1}l_{2}e_{2}(\lambda )+\delta _{1}e_{2}(\lambda l_{2}) 
\notag \\
&&-\delta _{1}\lambda \mu l_{2}-\delta _{1}\lambda ^{2}l_{1}+\mu \left\{
\delta _{2}l_{1}^{2}+\delta _{1}l_{2}^{2}-\delta _{1}\delta
_{2}K^{B}\right\} ,  \notag
\end{eqnarray}%
\textit{where }$K^{B}=R_{1221}^{B}\circ \pi =\delta _{2}e_{1}(l_{2})-\delta
_{1}e_{2}(l_{1})-\delta _{1}l_{1}^{2}-\delta _{2}l_{2}^{2}$\textit{\ is the
Gauss curvature of Riemannian manifold (}$B_{s}^{2},g^{^{\prime }})$.
\end{proposition}

\begin{proof}
Let \ $\nabla $ denote the Levi-Civita connection of the pseudo-Riemannian
manifold $(M_{r}^{3},g).$ Using (\ref{3}), Koszul formula and after a
straightforward computation, we have%
\begin{eqnarray}
\nabla _{e_{1}}e_{1} &=&-\delta _{1}\delta _{2}l_{1}e_{2,}\text{ \ \ }\nabla
_{e_{1}}e_{2}=l_{1}e_{1}-\sigma e_{3},\text{ \ }  \notag \\
\text{\ }\nabla _{e_{1}}e_{3} &=&\delta _{2}\delta _{3}\sigma
e_{2},~~~~~\nabla _{e_{2}}e_{1}=-l_{2}e_{2}+\sigma e_{3},  \notag \\
\text{ \ \ }\nabla _{e_{2}}e_{2} &=&\delta _{1}\delta _{2}l_{2}e_{1},\text{
\ ~~\ }\nabla _{e_{2}}e_{3}=-\delta _{1}\delta _{3}\sigma e_{1},  \label{5}
\\
\nabla _{e_{3}}e_{1} &=&\delta _{2}\delta _{3}\sigma e_{2}-\lambda e_{3},%
\text{ \ \ }\nabla _{e_{3}}e_{2}=-\delta _{1}\delta _{3}\sigma e_{1}-\mu
e_{3},\text{ }  \notag \\
\text{\ \ }\nabla _{e_{3}}e_{3} &=&\delta _{1}\delta _{3}\lambda
e_{1}+\delta _{2}\delta _{3}\mu e_{2}.  \notag
\end{eqnarray}%
The tension of the pseudo-Riemannian submersion $\tau $ is given by%
\begin{equation}
\tau (\pi )=\sum_{i=1}^{3}g(e_{i},e_{i})\left[ \nabla _{e_{i}}^{\pi }d\pi
(e_{i})-d\pi (\nabla _{e_{i}}^{M}e_{i})\right] =-\delta _{3}d\pi (\nabla
_{e_{3}}^{M}e_{3})=-\delta _{1}\lambda \varepsilon _{1}-\delta _{2}\mu
\varepsilon _{2}.  \label{6}
\end{equation}%
After some calculation by using (\ref{5}) we get%
\begin{eqnarray*}
\tau ^{2}(\pi ) &=&\sum_{i=1}^{3}g(e_{i},e_{i})\left\{ \nabla _{e_{i}}^{\pi
}\nabla _{e_{i}}^{\pi }\tau (\pi )-\nabla _{\nabla _{e_{i}}^{M}e_{i}}^{\pi
}\tau (\pi )-R^{B}(d\pi (e_{i}),\tau (\pi ))d\pi (e_{i})\right\} \\
&=&\delta _{1}\left[ 
\begin{array}{c}
\nabla _{e_{1}}^{\pi }(-\delta _{1}e_{1}(\lambda )\varepsilon _{1}-\delta
_{1}\lambda \nabla _{e_{1}}^{\pi }\varepsilon _{1})+\nabla _{e_{1}}^{\pi
}(-\delta _{2}e_{1}(\mu )\varepsilon _{2}-\delta _{2}\mu \nabla
_{e_{1}}^{\pi }\varepsilon _{2}) \\ 
+\delta _{1}\delta _{2}l_{1}\nabla _{e_{2}}^{\pi }(-\delta _{1}\lambda
\varepsilon _{1}-\delta _{2}\mu \varepsilon _{2})+\delta _{2}\mu
R^{B}(\varepsilon _{1},\varepsilon _{2})\varepsilon _{1}%
\end{array}%
\right] \\
&&+\delta _{2}\left[ 
\begin{array}{c}
\nabla _{e_{2}}^{\pi }(-\delta _{1}e_{2}(\lambda )\varepsilon _{1}-\delta
_{1}\lambda \nabla _{e_{2}}^{\pi }\varepsilon _{1})+\nabla _{e_{2}}^{\pi
}(-\delta _{2}e_{2}(\mu )\varepsilon _{2}-\delta _{2}\mu \nabla
_{e_{2}}^{\pi }\varepsilon _{2}) \\ 
-\delta _{1}\delta _{2}l_{2}\nabla _{e_{1}}^{\pi }(-\delta _{1}\lambda
\varepsilon _{1}-\delta _{2}\mu \varepsilon _{2})+\delta _{1}\lambda
R^{B}(\varepsilon _{2},\varepsilon _{1})\varepsilon _{2}%
\end{array}%
\right] \\
&&\delta _{3}\left[ 
\begin{array}{c}
\nabla _{e_{3}}^{\pi }(-\delta _{1}e_{3}(\lambda )\varepsilon _{1}-\delta
_{1}\lambda \nabla _{e_{3}}^{\pi }\varepsilon _{1})+\nabla _{e_{3}}^{\pi
}(-\delta _{2}e_{3}(\mu )\varepsilon _{2}-\delta _{2}\mu \nabla
_{e_{3}}^{\pi }\varepsilon _{2}) \\ 
-\delta _{1}\delta _{3}\lambda \nabla _{e_{1}}^{\pi }(-\delta _{1}\lambda
\varepsilon _{1}-\delta _{2}\mu \varepsilon _{2})-\delta _{2}\delta _{3}\mu
\nabla _{e_{2}}^{\pi }(-\delta _{1}\lambda \varepsilon _{1}-\delta _{2}\mu
\varepsilon _{2})%
\end{array}%
\right] .\text{ }
\end{eqnarray*}%
Now we calculate Laplace of $\lambda $ and $\mu $. Since $grad\lambda
=\delta _{1}e_{1}(\lambda )e_{1}+\delta _{2}e_{2}(\lambda )e_{2}+\delta
_{3}e_{3}(\lambda )e_{3}$, we obtain 
\begin{eqnarray*}
\Delta ^{m}\lambda &=&\dsum\limits_{i=1}^{3}g(e_{i},e_{i})g(\nabla
_{e_{i}}grad\lambda ,e_{i}) \\
&=&\delta _{1}e_{1}(e_{1}(\lambda ))+\delta _{2}e_{2}(e_{2}(\lambda
))+\delta _{3}e_{3}(e_{3}(\lambda ))+\delta _{2}e_{2}(\lambda )l_{1}-\delta
_{1}e_{1}(\lambda )l_{2} \\
&&-\delta _{1}e_{1}(\lambda )\lambda -\delta _{2}e_{2}(\lambda )\mu .
\end{eqnarray*}%
Using same calculations for $\mu $ we get 
\begin{eqnarray*}
\Delta ^{m}\mu &=&\delta _{1}e_{1}(e_{1}(\mu ))+\delta _{2}e_{2}(e_{2}(\mu
))+\delta _{3}e_{3}(e_{3}(\mu ))+\delta _{2}e_{2}(\mu )l_{1}-\delta
_{1}e_{1}(\mu )l_{2} \\
&&-\delta _{1}e_{1}(\mu )\lambda -\delta _{2}e_{2}(\mu )\mu .
\end{eqnarray*}%
\begin{eqnarray*}
\tau ^{2}(\pi ) &=&\delta _{1}\left[ 
\begin{array}{c}
-\Delta ^{M}\lambda -\delta _{2}l_{1}e_{1}(\mu )-\delta _{2}e_{1}(\mu
l_{1})-\delta _{2}l_{2}e_{2}(\mu )-\delta _{2}e_{2}(\mu l_{2}) \\ 
+\delta _{2}\lambda \mu l_{1}+\delta _{2}\mu ^{2}l_{2}+\lambda \left\{
\delta _{2}l_{1}^{2}+\delta _{1}l_{2}^{2}-\delta _{1}\delta _{2}K^{B}\right\}%
\end{array}%
\right] \varepsilon _{1} \\
&&+\delta _{2}\left[ 
\begin{array}{c}
-\Delta ^{M}\mu +\delta _{1}l_{1}e_{1}(\lambda )+\delta _{1}e_{1}(\lambda
l_{1})+\delta _{1}l_{2}e_{2}(\lambda )+\delta _{1}e_{2}(\lambda l_{2}) \\ 
-\delta _{1}\lambda \mu l_{2}-\delta _{1}\lambda ^{2}l_{1}+\mu \left\{
\delta _{2}l_{1}^{2}+\delta _{1}l_{2}^{2}-\delta _{1}\delta _{2}K^{B}\right\}%
\end{array}%
\right] \varepsilon _{2,}
\end{eqnarray*}%
which completes the proof.

When the integrability function $\mu =0$ we have the following corollary.

\begin{corollary}
Let $\pi :(M_{r}^{3},g)\rightarrow (B_{s}^{2},g^{^{\prime }})$ be a
pseudo-Riemannian submersion with an adapted frame $\left\{
e_{1},e_{2},e_{3}\right\} $ and the integrability functions $l_{1}$, $l_{2}$%
, $\lambda $, $\mu $ and \ $\sigma $ with $\mu =0$ . Then, the
pseudo-Riemannian submersion $\pi $ is biharmonic if and only if%
\begin{eqnarray}
-\delta _{1}\Delta ^{M}\lambda +\lambda \left\{ \delta _{1}\delta
_{2}l_{1}^{2}+l_{2}^{2}-\delta _{2}K^{B}\right\} &=&0,  \label{7} \\
\delta _{1}\delta _{2}l_{1}e_{1}(\lambda )+\delta _{1}\delta
_{2}e_{1}(\lambda l_{1})+\delta _{1}\delta _{2}l_{2}e_{2}(\lambda )+\delta
_{1}\delta _{2}e_{2}(\lambda l_{2})-\delta _{1}\delta _{2}\lambda ^{2}l_{1}
&=&0.  \notag
\end{eqnarray}
\end{corollary}
\end{proof}

The following lemmas will be used to prove Classification Theorem.

\begin{lemma}
Let $\pi :M_{r}^{3}(c)\rightarrow (B_{s}^{2},g^{^{\prime }})$ be a
pseudo-Riemannian submersion from a space form of constant sectional
curvature $c$. Then, for any orthonormal frame $\left\{
e_{1},e_{2},e_{3}\right\} $ on $M_{r}^{3}(c)$ adapted to the
pseudo-Riemannian submersion with $e_{3\text{ }}$being vertical, all the
integrability functions $l_{1}$, $l_{2}$, $\lambda $, $\mu $ and \ $\sigma $
are constant along fibers of $\pi $, i.e.,%
\begin{equation}
e_{3}(l_{1})=e_{3}(l_{2})=e_{3}(\mu )=e_{3}(\lambda )=e_{3}(\sigma )=0
\label{9}
\end{equation}
\end{lemma}

\begin{proof}
From definition, $l_{i}=F_{i}\circ \pi $ for $i=1,2$ we can conclude that $%
l_{1}$ and $l_{2}$ are constant along the fibers. It remains to show that 
\begin{equation}
e_{3}(\mu )=e_{3}(\lambda )=e_{3}(\sigma )=0.  \label{9,6}
\end{equation}%
Using the Jacobi identity to the frame $\left\{ e_{1},e_{2},e_{3}\right\} $,
we have%
\begin{equation}
2e_{3}(\sigma )+\lambda l_{1}+\mu l_{2}+e_{2}(\lambda )-e_{1}(\mu )=0.
\label{11,5}
\end{equation}%
By using (\ref{11,5}) and the fact that $M_{1}^{3}(c)$ has constant
sectional curvature $c$, calculating $R_{1312}^{M},$ $R_{1313}^{M},$ $%
R_{1323}^{M},$ $R_{1212}^{M},$ $R_{1223}^{M},$ $R_{2313}^{M},$ $R_{2323}^{M}$
respectively, we get 
\begin{eqnarray}
i)e_{1}(\sigma )-2\lambda \sigma &=&0,  \notag \\
ii)\text{ }\delta _{1}e_{1}(\lambda )+\delta _{1}\delta _{2}\delta
_{3}\sigma ^{2}-\delta _{1}\lambda ^{2}+\delta _{2}\mu l_{1} &=&c,  \notag \\
iii)-e_{1}(\mu )+e_{3}(\sigma )+\lambda l_{1}+\lambda \mu &=&0,  \notag \\
iv)-\delta _{2}e_{2}(l_{1})+\delta _{1}e_{1}(l_{2})-\delta
_{2}l_{1}^{2}-\delta _{1}l_{2}^{2}-3\delta _{1}\delta _{2}\delta _{3}\sigma
^{2} &=&c,  \label{12} \\
v)e_{2}(\sigma )-2\mu \sigma &=&0,  \notag \\
vi)-e_{2}(\lambda )-e_{3}(\sigma )-\mu l_{2}+\lambda \mu &=&0,  \notag \\
vii)\text{\ }\delta _{1}\delta _{2}\delta _{3}\sigma ^{2}+\delta
_{2}e_{2}(\mu )-\delta _{1}\lambda l_{2}-\delta _{2}\mu ^{2} &=&c.  \notag
\end{eqnarray}%
Applying $e_{3}$ to both sides of the equation $iv)$ of (\ref{12}) and using
\ $e_{3}e_{1}=\left[ e_{3},e_{1}\right] +e_{1}e_{3}$ and $e_{3}e_{2}=\left[
e_{3},e_{2}\right] +e_{2}e_{3},$ we obtain%
\begin{equation*}
\sigma e_{3}(\sigma )=0,
\end{equation*}%
which implies%
\begin{equation*}
e_{3}(\sigma )=0.
\end{equation*}%
Using the last equation and applying $e_{3}$ to both sides of the equations $%
i)$ and $v)$ of (\ref{12}) respectively, we get%
\begin{equation*}
e_{3}(\lambda )=0,\text{ \ \ }e_{3}(\mu )=0.
\end{equation*}
\end{proof}

\begin{case}
Spacelike Fiber
\end{case}

$%
\begin{tabular}{|c|c|}
\hline
\begin{tabular}{c}
Submersion \\ 
Signature of $g$ \\ 
Signature of $g^{\prime }$%
\end{tabular}
& \ \ \ \ \ \ \ \ \ \ \ \ New Orthonormal frame of Base Manifold \\ \hline
\begin{tabular}{c}
$\pi :(M_{1}^{3},g)\rightarrow (B_{1}^{2},g^{\prime })$ \\ 
$(e_{1},e_{2},e_{3};+,-,+)$ \\ 
$(\varepsilon _{1},\varepsilon _{2};+,-)$%
\end{tabular}
& $%
\begin{tabular}{l}
$\varepsilon _{1}^{^{\prime }}=-\frac{\bar{\lambda}}{\sqrt{\bar{\lambda}^{2}-%
\bar{\mu}^{2}}}\varepsilon _{1}+\frac{\bar{\mu}}{\sqrt{\bar{\lambda}^{2}-%
\bar{\mu}^{2}}}\varepsilon _{2},$ $\varepsilon _{2}^{^{\prime }}=-\frac{\bar{%
\mu}}{\sqrt{\bar{\lambda}^{2}-\bar{\mu}^{2}}}\varepsilon _{1}+\frac{\bar{%
\lambda}}{\sqrt{\bar{\lambda}^{2}-\bar{\mu}^{2}}}\varepsilon _{2};$if $\bar{%
\lambda}^{2}-\bar{\mu}^{2}>0$ \\ 
$\varepsilon _{1}^{^{\prime }}=-\frac{\bar{\mu}}{\sqrt{\bar{\mu}^{2}-\bar{%
\lambda}^{2}}}\varepsilon _{1}+\frac{\bar{\lambda}}{\sqrt{\bar{\mu}^{2}-\bar{%
\lambda}^{2}}}\varepsilon _{2},$ $\varepsilon _{2}^{^{\prime }}=-\frac{\bar{%
\lambda}}{\sqrt{\bar{\mu}^{2}-\bar{\lambda}^{2}}}\varepsilon _{1}+\frac{\bar{%
\mu}}{\sqrt{\bar{\mu}^{2}-\bar{\lambda}^{2}}}\varepsilon _{2};$if $\bar{\mu}%
^{2}-\bar{\lambda}^{2}>0$%
\end{tabular}%
$ \\ \hline
\begin{tabular}{c}
$\pi :(M_{2}^{3},g)\rightarrow (B_{2}^{2},g^{\prime })$ \\ 
$(e_{1},e_{2},e_{3};-,-,+)$ \\ 
$(\varepsilon _{1},\varepsilon _{2};-,-)$%
\end{tabular}
& $\varepsilon _{1}^{^{\prime }}=\frac{\bar{\lambda}}{\sqrt{\bar{\lambda}%
^{2}+\bar{\mu}^{2}}}\varepsilon _{1}+\frac{\bar{\mu}}{\sqrt{\bar{\lambda}%
^{2}+\bar{\mu}^{2}}}\varepsilon _{2},\varepsilon _{2}^{^{\prime }}=\frac{%
\bar{\mu}}{\sqrt{\bar{\lambda}^{2}+\bar{\mu}^{2}}}\varepsilon _{1}-\frac{%
\bar{\lambda}}{\sqrt{\bar{\lambda}^{2}+\bar{\mu}^{2}}}\varepsilon _{2}$ \\ 
\hline
\begin{tabular}{c}
$\pi :(M^{3},g)\rightarrow (B^{2},g^{\prime })$ \\ 
$(e_{1},e_{2},e_{3};+,+,+)$ \\ 
$(\varepsilon _{1},\varepsilon _{2};+,+)$%
\end{tabular}
& $\varepsilon _{1}^{^{\prime }}=\frac{\bar{\lambda}}{\sqrt{\bar{\lambda}%
^{2}+\bar{\mu}^{2}}}\varepsilon _{1}+\frac{\bar{\mu}}{\sqrt{\bar{\lambda}%
^{2}+\bar{\mu}^{2}}}\varepsilon _{2},$ $\varepsilon _{2}^{^{\prime }}=-\frac{%
\bar{\mu}}{\sqrt{\bar{\lambda}^{2}+\bar{\mu}^{2}}}\varepsilon _{1}+\frac{%
\bar{\lambda}}{\sqrt{\bar{\lambda}^{2}+\bar{\mu}^{2}}}\varepsilon _{2}$ \\ 
\hline
\end{tabular}%
$

\begin{equation*}
Table\text{ }1
\end{equation*}

\begin{case}
Timelike Fiber
\end{case}

\begin{tabular}{|c|c|}
\hline
\begin{tabular}{c}
Submersion \\ 
Signature of $g$ \\ 
Signature of $g^{\prime }$%
\end{tabular}
& \ \ \ \ \ \ \ \ \ \ \ \ New Orthonormal frame of Base Manifold \\ \hline
\begin{tabular}{c}
$\pi :(M_{1}^{3},g)\rightarrow (B^{2},g^{\prime })$ \\ 
$(e_{1},e_{2},e_{3};+,+,-)$ \\ 
$(\varepsilon _{1},\varepsilon _{2};+,+)$%
\end{tabular}
& $\varepsilon _{1}^{^{\prime }}=\frac{\bar{\lambda}}{\sqrt{\bar{\lambda}%
^{2}+\bar{\mu}^{2}}}\varepsilon _{1}+\frac{\bar{\mu}}{\sqrt{\bar{\lambda}%
^{2}+\bar{\mu}^{2}}}\varepsilon _{2},\varepsilon _{2}^{^{\prime }}=\frac{%
\bar{\mu}}{\sqrt{\bar{\lambda}^{2}+\bar{\mu}^{2}}}\varepsilon _{1}-\frac{%
\bar{\lambda}}{\sqrt{\bar{\lambda}^{2}+\bar{\mu}^{2}}}\varepsilon _{2}$ \\ 
\hline
\begin{tabular}{c}
$\pi :(M_{2}^{3},g)\rightarrow (B_{1}^{2},g^{\prime })$ \\ 
$(e_{1},e_{2},e_{3};+-,-)$ \\ 
$(\varepsilon _{1},\varepsilon _{2}:+,-)$%
\end{tabular}
& $%
\begin{tabular}{l}
$\varepsilon _{1}^{^{\prime }}=-\frac{\bar{\lambda}}{\sqrt{\bar{\lambda}^{2}-%
\bar{\mu}^{2}}}\varepsilon _{1}+\frac{\bar{\mu}}{\sqrt{\bar{\lambda}^{2}-%
\bar{\mu}^{2}}}\varepsilon _{2},$ $\varepsilon _{2}^{^{\prime }}=-\frac{\bar{%
\mu}}{\sqrt{\bar{\lambda}^{2}-\bar{\mu}^{2}}}\varepsilon _{1}+\frac{\bar{%
\lambda}}{\sqrt{\bar{\lambda}^{2}-\bar{\mu}^{2}}}\varepsilon _{2};$if $\bar{%
\lambda}^{2}-\bar{\mu}^{2}>0$ \\ 
$\varepsilon _{1}^{^{\prime }}=-\frac{\bar{\mu}}{\sqrt{\bar{\mu}^{2}-\bar{%
\lambda}^{2}}}\varepsilon _{1}+\frac{\bar{\lambda}}{\sqrt{\bar{\mu}^{2}-\bar{%
\lambda}^{2}}}\varepsilon _{2},$ $\varepsilon _{2}^{^{\prime }}=-\frac{\bar{%
\lambda}}{\sqrt{\bar{\mu}^{2}-\bar{\lambda}^{2}}}\varepsilon _{1}+\frac{\bar{%
\mu}}{\sqrt{\bar{\mu}^{2}-\bar{\lambda}^{2}}}\varepsilon _{2};$if $\bar{\mu}%
^{2}-\bar{\lambda}^{2}>0$%
\end{tabular}%
$ \\ \hline
\begin{tabular}{c}
$\pi :(M_{3}^{3},g)\rightarrow (B_{2}^{2},g^{\prime })$ \\ 
$(e_{1},e_{2},e_{3};-,-,-)$ \\ 
$(\varepsilon _{1},\varepsilon _{2}:-,-)$%
\end{tabular}
& $\varepsilon _{1}^{^{\prime }}=\frac{\bar{\lambda}}{\sqrt{\bar{\lambda}%
^{2}+\bar{\mu}^{2}}}\varepsilon _{1}+\frac{\bar{\mu}}{\sqrt{\bar{\lambda}%
^{2}+\bar{\mu}^{2}}}\varepsilon _{2},\varepsilon _{2}^{^{\prime }}=\frac{%
\bar{\mu}}{\sqrt{\bar{\lambda}^{2}+\bar{\mu}^{2}}}\varepsilon _{1}-\frac{%
\bar{\lambda}}{\sqrt{\bar{\lambda}^{2}+\bar{\mu}^{2}}}\varepsilon _{2}$ \\ 
\hline
\end{tabular}%
\begin{equation*}
Table\text{ }2
\end{equation*}

\begin{lemma}
Let $\pi :(M_{r}^{3}(c),g)\rightarrow (B_{s}^{2},g^{^{\prime }})$ be a
pseudo-Riemannian submersion with an adapted frame $\left\{
e_{1},e_{2},e_{3}\right\} $ and the integrability functions $l_{1}$, $l_{2}$%
, $\lambda $, $\mu $ and \ $\sigma $ . Then, there exists another adapted
orthonormal frame $\left\{ e_{1}^{^{\prime }},e_{2}^{^{\prime
}},e_{3}^{^{\prime }}=e_{3}\right\} $ on $M_{r}^{3}(c)$ with integrability
functions $\mu ^{^{\prime }}=0,$ and $\sigma ^{^{\prime }}=\sigma .$
\end{lemma}

\begin{proof}
Applying the same method in (\cite{ze-ye}, Lemma 3.2) and using Lemma 4 ,
Table 1 and Table 2, one can complete the proof of the lemma.
\end{proof}

Now we will give a classification of biharmonic pseudo-Riemannian
submersions.

\textbf{Classification Theorem:}\textit{\ Let }$\pi :M_{r}^{3}(c)\rightarrow
B_{s}^{2}$\textit{\ be a pseudo-Riemannian submersion from a space form of
constant sectional curvature }$c$\textit{. Then, }$\pi $\textit{\ is
biharmonic if and only if it is equivalent to one of the following
submersions:}

\begin{tabular}{|l|l|}
\hline
\ \ \ \ \ \ \ \ \ \ Timelike Fiber & \ \ \ \ \ \ \ \ \ \ \ \ \ \ \ \ \ \ \ \
\ Spacelike Fiber \\ \hline
$\pi _{1}:H_{3}^{3}(-1)\rightarrow H_{2}^{2}(-4)=CH_{1}^{1};$ & $\pi
_{6}:E_{2}^{3}\rightarrow E_{2}^{2};$ \\ \hline
$\pi _{2}:E_{3}^{3}\rightarrow E_{2}^{2};$ & $\pi
_{7}:H_{1}^{3}(-1)\rightarrow H_{1}^{2}(-4)=AH^{1};$ \\ \hline
$\pi _{3}:H_{1}^{3}(-1)\rightarrow H^{2}(-4)=CH^{1};$ & $\pi
_{8}:E_{1}^{3}\rightarrow E_{1}^{2};$ \\ \hline
$\pi _{4}:E_{1}^{3}\rightarrow E^{2};$ & $\pi _{9}:S^{3}(1)\rightarrow
S^{2}\left( \frac{1}{2}\right) =CP^{1};$is proved by \cite{ze-ye} \\ \hline
$\pi _{5}:E_{2}^{3}\rightarrow E_{1}^{2};$ & $\pi _{10}:E^{3}\rightarrow
E^{2},$is proved by \cite{ze-ye} \\ \hline
\end{tabular}%
\begin{equation*}
Table\text{ }3
\end{equation*}

\begin{proof}
By Lemma $5$, we can choose an orthonormal frame $\left\{
e_{1},e_{2},e_{3}\right\} $ adapted to the pseudo-Riemannian submersion with
integrability functions $l_{1}$, $l_{2}$, $\lambda $, $\mu $ and \ $\sigma $
with $\mu =0$. According to this frame (\ref{12}) reduces to%
\begin{eqnarray}
a_{1})e_{1}(\sigma )-2\lambda \sigma &=&0,  \notag \\
a_{2})\delta _{1}e_{1}(\lambda )+\delta _{1}\delta _{2}\delta _{3}\sigma
^{2}-\delta _{1}\lambda ^{2} &=&c,  \notag \\
a_{3})\lambda l_{1} &=&0,  \notag \\
a_{4})-\delta _{2}e_{2}(l_{1})+\delta _{1}e_{1}(l_{2})-\delta
_{2}l_{1}^{2}-\delta _{1}l_{2}^{2}-3\delta _{1}\delta _{2}\delta _{3}\sigma
^{2} &=&c,  \label{13} \\
a_{5})e_{2}(\sigma ) &=&0,  \notag \\
a_{6})e_{2}(\lambda ) &=&0,  \notag \\
a_{7})\delta _{1}\delta _{2}\delta _{3}\sigma ^{2}-\delta _{1}\lambda l_{2}
&=&c.  \notag
\end{eqnarray}

From $a_{3})$ of (\ref{13}), we have either $\lambda =0$ or $l_{1}=0.$ If $%
\lambda =0$, from (\ref{6}) the tension field of $\pi $ vanishes. This means
that pseudo-Riemannian submersion is harmonic. If $l_{1}=0$ and $\lambda
\neq 0$, this case can not happen. We will prove this by a contradiction.

Case I$:$ $\lambda \neq 0,$ $\ l_{1}=0$ and $\ l_{2}=0.$ So, from $a_{4}),$ $%
a_{7})$ in (\ref{13}), we have $\sigma =c=0$. If we put $l_{1}=l_{2}=\sigma
=0$ and $\mu =0$ into (\ref{7}) we obtain%
\begin{equation*}
\Delta ^{M}\lambda =0,
\end{equation*}%
which, one can easily get by using $a_{2}),$ $a_{6})$ of (\ref{13}) ,%
\begin{equation*}
\lambda ^{3}=0.
\end{equation*}%
It follows that $\lambda =0$ which is a contradiction.

Case II$:$ $\lambda \neq 0,$ $\ \ l_{1}=0$ and $\ l_{2}\neq 0.$ In this
case, by using $\ l_{1}=0$ \ and $a_{5}),$ $a_{6})$ and $a_{7})$ of (\ref{13}%
), (\ref{7}) reduces to%
\begin{equation}
-\delta _{1}\Delta ^{M}\lambda +\lambda \left[ -\delta _{2}c-3\delta
_{1}\delta _{3}\sigma ^{2}+l_{2}^{2}\right] =0,  \label{144}
\end{equation}%
where $K^{B}=c+3\delta _{1}\delta _{2}\delta _{3}\sigma ^{2}$ obtained from
curvature formula for a pseudo-Riemannian submersion. Using $a_{1}),a_{2})$
of (\ref{13}) and after a straightforward calculation yields%
\begin{eqnarray*}
\Delta ^{M}\lambda &=&\delta _{1}e_{1}(e_{1}(\lambda ))-\delta
_{1}e_{1}(\lambda )l_{2}-\delta _{1}e_{1}(\lambda )\lambda \\
\Delta ^{M}\lambda &=&-5\delta _{1}\delta _{2}\delta _{3}\lambda \sigma
^{2}+\delta _{1}\lambda ^{3}+\lambda c+l_{2}(-c+\delta _{1}\delta _{2}\delta
_{3}\sigma ^{2}-\delta _{1}\lambda ^{2}).
\end{eqnarray*}%
Substituting this into (\ref{144}) and using $a_{7})$ we obtain%
\begin{equation}
\lambda \left[ \delta _{3}(6\delta _{2}-3\delta _{1})\sigma ^{2}-\lambda
^{2}-(2\delta _{1}+\delta _{2})c\right] =0.  \label{15}
\end{equation}%
We accept $\lambda \neq 0,$ so (\ref{15}) is equivalent to \ 
\begin{equation}
\lambda ^{2}=\delta _{3}(6\delta _{2}-3\delta _{1})\sigma ^{2}-(2\delta
_{1}+\delta _{2})c.  \label{16}
\end{equation}%
After applying $e_{1}$ to both sides of (\ref{16}), we get 
\begin{equation*}
\lambda e_{1}(\lambda )=\delta _{3}(6\delta _{2}-3\delta _{1})\sigma
e_{1}(\sigma ).
\end{equation*}%
Combining this and $a_{1})$ , $a_{2})$ in (\ref{13}), we have%
\begin{equation*}
\lambda (\lambda ^{2}-\delta _{2}\delta _{3}\sigma ^{2}+\delta
_{1}c)=2\delta _{3}(6\delta _{2}-3\delta _{1})\lambda \sigma ^{2}.
\end{equation*}%
By assumption $\lambda \neq 0,$ this turned into%
\begin{equation*}
\lambda ^{2}+\delta _{1}c=\delta _{3}(13\delta _{2}-6\delta _{1})\sigma ^{2},
\end{equation*}%
or%
\begin{equation}
\lambda ^{2}=\delta _{3}(13\delta _{2}-6\delta _{1})\sigma ^{2}-\delta _{1}c.
\label{17}
\end{equation}%
Applying $e_{1}$ to both sides of (\ref{17}) \ and again using $a_{1})$, $%
a_{2})$ in (\ref{13}) we get%
\begin{equation}
\lambda ^{2}=\delta _{3}(27\delta _{2}-12\delta _{1})\sigma ^{2}-\delta
_{1}c.  \label{18}
\end{equation}%
Combining (\ref{16}), (\ref{17})\ with (\ref{18}) we have $\lambda =\sigma
=c=0.$ This implies there is a contradiction. Because our assumption is $%
\lambda \neq 0.$So we have $\ \lambda =\mu =0.$ If we use (\ref{5}) in the
first equation of (\ref{AT}) we get $T(e_{i},e_{j})=0,$ $1\leq i,j\leq 3.$
It means that fiber is totally geodesic. By (a$_{2})$of (\ref{13}), we \ have%
\begin{equation}
\delta _{1}\delta _{2}\delta _{3}\sigma ^{2}=c.  \label{19}
\end{equation}

Using the last equation and Theorem 3 , we get our classification.
\end{proof}

\end{document}